\documentclass[reqno,11pt]{amsart}

\usepackage{a4wide}
\usepackage[english]{babel}
\usepackage[utf8]{inputenc}
\usepackage{amsmath,amssymb}
\usepackage{dsfont}
\usepackage{tikz}
\usepackage[hidelinks]{hyperref}

\theoremstyle{definition}
\newtheorem{thm}{Theorem}[section]
\newtheorem{lem}[thm]{Lemma}
\newtheorem{prop}[thm]{Proposition}

\newtheorem{defi}[thm]{Definition}
\newtheorem{rem}[thm]{Remark}
\numberwithin{equation}{section}

\renewcommand{\Im}{\operatorname{Im}}
\newcommand{\ran}{\operatorname{ran}}
\newcommand{\dom}{\operatorname{dom}}
\newcommand{\poly}{{\text{poly}}}

\title[]{The $H^\infty$-functional calculus for right slice hyperholomorphic functions and right linear Clifford operators}

\author[Fabrizio Colombo]{F. Colombo}
\address{(FC) Politecnico di Milano, Dipartimento di Matematica, Via E. Bonardi 9, 20133 Milano, Italy}
\email{fabrizio.colombo@polimi.it}

\author[Francesco Mantovani]{F. Mantovani}
\address{(FM) Politecnico di Milano, Dipartimento di Matematica, Via E. Bonardi 9, 20133 Milano, Italy}
\email{francesco.mantovani@polimi.it}

\author[P. Schlosser]{P. Schlosser}
\address{(PS) Graz University of Technology, Institut f\"ur Angewandte Mathematik, Steyrergasse 30, 8010 Graz, Austria}
\email{pschlosser@math.tugraz.at}
\begin{document}

\begin{abstract}
In 2016, the spectral theory on the $S$-spectrum was used to establish the $H^\infty$-functional calculus for quaternionic or Clifford operators. This calculus applies for example to sectorial or bisectorial right linear operators $T$ and left slice hyperholomorphic functions $f$ that can grow as polynomials. It relies on the product of the two operators $e(T)^{-1}$ and $(ef)(T)$, both defined via some underlying $S$-functional calculus (also called $\omega$-functional calculus). For left slice holomorphic functions $f$ this definition does not depend on the choice of the regularizer function $e$. However, due to the non-commutative multiplication of Clifford numbers, it was unclear how to extend this definition to right slice hyperholomorphic functions. This paper addresses this significant unresolved issue and shows how right linear operators can possess the $H^\infty$-functional calculus also for right slice hyperholomorphic functions.
\end{abstract}

\maketitle

 \medskip

 \medskip

\textbf{Acknowledgements:} Fabrizio Colombo is supported by MUR grant Dipartimento di Eccellenza 2023-2027. Peter Schlosser was funded by the Austrian Science Fund (FWF) under Grant No. J 4685-N and by the European Union--NextGenerationEU.

\section{Introduction}

The $H^\infty$-functional calculus applies to sectorial and bisectorial operators, and extends the holomorphic functional calculus, often referred to as the Riesz-Dunford functional calculus \cite{RD}, to functions that do not decay sufficiently fast at zero and infinity. For a comprehensive exploration of this calculus, we recommend the books \cite{Haase,HYTONBOOK1,HYTONBOOK2} along with their associated references. \medskip

The spectral theory on the $S$-spectrum has fundamental differences compared to classical complex spectral theory, because it is derived from the classical resolvent series expansion, without assuming that the operator and the spectral parameter commute. Specifically, for bounded operators $T:V\rightarrow V$ acting in a Clifford module $V$, the series expansion of the $S$-resolvent becomes
\begin{equation}\label{Eq_Resolvent_sum}
\sum_{n=0}^\infty T^ns^{-n-1}=(T^2-2s_0T+|s|^2)^{-1}(\overline{s}-T),\qquad|s|>\Vert T\Vert,
\end{equation}
where $s=s_0+s_1e_1+\dots+s_ne_n$ is a paravector, $|s|$ is its modulus and $\overline{s}$ the conjugate of $s$. The reason why the value of the series \eqref{Eq_Resolvent_sum} does not simplify to the classical resolvent operator $(s-T)^{-1}$ is due to the noncommutativity $sT\neq Ts$. The explicit value of the sum in equation \eqref{Eq_Resolvent_sum} motivates that, even for unbounded, closed operators $T$, the spectrum must be associated with the invertibility of the operator
\begin{equation*}
Q_s[T]:=T^2-2s_0T+|s|^2,\qquad\text{with }\dom(Q_s[T])=\dom(T^2).
\end{equation*}
This leads us to the definitions of the \textit{$S$-resolvent set} and the \textit{$S$-spectrum}
\begin{equation*}
\rho_S(T):=\big\{s\in\mathbb{R}^{n+1}\;\big|\;Q_s[T]^{-1}\in\mathcal{B}(V)\big\}\qquad\text{and}\qquad\sigma_S(T):=\mathbb{R}^{n+1}\setminus\rho_S(T),
\end{equation*}
where $\mathcal{B}(V)$ is the space of all bounded operators from $V$ into itself (see \cite{CGK,ColomboSabadiniStruppa2011} for more details). Motivated by \eqref{Eq_Resolvent_sum}, we define for every $s\in\rho_S(T)$ the \textit{left} and the \textit{right $S$-resolvent operator}
\begin{equation}\label{Eq_SL_SR}
S_L^{-1}(s,T):=Q_s[T]^{-1}\overline{s}-TQ_s[T]^{-1}\qquad\text{and}\qquad S_R^{-1}(s,T):=(\overline{s}-T)Q_s[T]^{-1}.
\end{equation}
which is then be used in the definition of the $S$-functional calculus
\begin{equation}\label{Eq_S_functional_calculus}
f(T):=\frac{1}{2\pi}\int_{\partial U\cap\mathbb{C}_J}S_L^{-1}(s,T)ds_Jf(s)\quad\text{and}\quad f(T):=\frac{1}{2\pi}\int_{\partial U\cap\mathbb{C}_J}f(s)ds_JS_R^{-1}(s,T),
\end{equation}
for left resp. right slice holomorphic functions $f$. This is the Clifford analogue of the Riesz-Dunford functional calculus. Moreover, we remark that the spectral theorem has recently been extended from quaternionic operators in \cite{ACK} to fully Clifford operators in \cite{ColKim}. \medskip

The $H^\infty$-functional calculus, as an extension of the $S$-functional calculus \eqref{Eq_S_functional_calculus}, has been first explored in works such as \cite{ACQS2016,CGdiffusion2018}, while the recent paper \cite{MS24} introduced the $H^\infty$-functional calculus for unbounded bisectorial operators within a Clifford module over the algebra $\mathbb{R}_n$. This work leverages the universality property of the $S$-functional calculus, as explained in \cite{ADVCGKS}, which demonstrates its applicability to fully Clifford operators. The idea of the $H^\infty$-functional calculus for left slice holomorphic functions $f$, is to find a regularizer function $e$, such that we can define
\begin{equation*}
f(T):=e(T)^{-1}(ef)(T),
\end{equation*}
where both $e(T)$ and $(ef)(T)$ are interpreted in the sense of the $S$-functional calculus \eqref{Eq_S_functional_calculus}. Thanks to \cite[Theorem 7.2.6]{FJBOOK}, this definition is independent of the regularizer $e$. Since right slice hyperholomorphic functions $f$ maintain slice hyperholomorphicity under multiplication with intrinsic functions $e$ from the right, a natural attempt of defining $f(T)$ would be as $(fe)(T)e(T)^{-1}$. However, it was already noted in \cite[Remark 7.2.2]{FJBOOK}, that this definition is not independent of the regularizer function $e$. For left linear operators, the situation is of course the opposite. \medskip

In Section~\ref{sec_Right_Hinfty}, we provide a solution to this longstanding open problem, by defining the $H^\infty$-functional calculus for right linear, bisectorial, injective operators $T$ and right slice holomorphic functions $f$ with at most polynomial growth, as the closure
\begin{equation}\label{Eq_Hinfty}
f(T):=\overline{(fe)(T)e(T)^{-1}}.
\end{equation}
Since we a priori do not know that $(fe)(T)e(T)^{-1}$ is closable, the closure in \eqref{Eq_Hinfty} is understood in the sense of multivalued operators, see Definition~\ref{defi_Multivalued_operators}. The independence of \eqref{Eq_Hinfty} from the regularizer $e$ is proven in Theorem~\ref{thm_Hinfty_independence} and crucially depends on the fact that for bisectorial operators the Banach module $V$ decomposes into $V=\ker(T)+\overline{\ran}(T)$, see Theorem~\ref{thm_Domain_dense}. The fact that the right $H^\infty$-functional calculus \eqref{Eq_Hinfty} coincides with the already existing left $H^\infty$-functional calculus for intrinsic functions, and with the right $\omega$-functional calculus for decaying functions, is proven in Theorem~\ref{thm_fT_as_closure} and Theorem~\ref{thm_Right_omega_as_closure}. \medskip

Quaternionic and Clifford operators play an important role across various fields of mathematics and physics. For example, they are appear in quaternionic quantum mechanics (see S. Adler's seminal work \cite{adler}), in vector analysis, differential geometry, and hypercomplex analysis.

\medskip
\textit{Vector analysis.} The gradient operator with nonconstant coefficients in $n$ dimensions represents various physical laws, such as Fourier's law for heat propagation and Fick's law for mass transfer diffusion. This can be expressed as
\begin{equation*}
T=\sum_{i=1}^ne_ia_i(x)\partial_{x_i},\qquad x\in\Omega,
\end{equation*}
and is associated with various boundary conditions, as discussed in \cite{CMS24}. Here, $e_1,\dots,e_n$ represent the imaginary units of the Clifford algebra $\mathbb{R}_n$, and the coefficients $a_1,\dots,a_n$ are assumed to belong to $C^1(\overline{\Omega})$ and satisfy appropriate bounds. \medskip

\textit{Differential geometry.} Let $(M,g)$ be a Riemannian manifold of dimension $n$, and $U$ a coordinate neighbourhood of $M$. In this neighbourhood, we can find an orthonormal frame of vector fields $E_1,\dots,E_n$, such that the Dirac operator on $(U,g)$ can be expressed as the $\mathcal{C}^\infty(U,\mathcal{H})$ differential operator
\begin{equation}\label{Eq_Dirac_manifold}
\mathcal{D}=\sum_{i=1}^ne_i\nabla_{E_i}^\tau,
\end{equation}
where $\nabla_{E_i}^\tau$ denotes the covariant derivative with respect to the vector field $E_i$. For more details see the book \cite{DiracHarm}. As a special case, \cite{DIRACHYPSPHE} considers the Dirac operator in hyperbolic and spherical spaces, where \eqref{Eq_Dirac_manifold} takes the explicit forms
\begin{equation*}
\mathcal{D}_H=x_n\sum_{i=1}^ne_i\partial_{x_i}-\frac{n-1}{2}e_n\qquad\text{and}\qquad\mathcal{D}_S=(1+|x|^2)\sum_{i=1}^ne_i\partial_{x_i}-nx,
\end{equation*}
for $(x_1,...,x_{n-1},x_n)\in\mathbb{R}^{n-1}\times\mathbb{R}^+$ and $(x_1,...,x_n)\in\mathbb{R}^n$, respectively. \medskip

\textit{Hypercomplex analysis.} The well-known Dirac operator and its conjugate
\begin{equation*}
D=\partial_{x_0}+\sum_{i=1}^ne_i\partial_{x_i}\qquad\text{and}\qquad\overline{D}=\partial_{x_0}-\sum_{i=1}^ne_i\partial_{x_i},
\end{equation*}
are widely investigated in \cite{DSS} as well as in \cite{DiracHarm}. Slice hyperholomorphic functions can also be viewed as functions in the kernel of the global operator
\begin{equation*}
G=|\Im(x)|^2\partial_{x_0}+\Im(x)\sum_{i=1}^nx_i\partial_{x_i},
\end{equation*}
introduced in \cite{6Global}, where $\Im(x)=x_1e_1+\dots+x_ne_n$. For integers $\alpha,\beta,m$, the operators of the Dirac fine structure on the $S$-spectrum are defined as
\begin{equation*}
T_{\alpha,m}=D^\alpha(D\overline{D})^m\qquad\text{and}\qquad\widetilde{T}_{\beta,m}=\overline{D}^\beta(D\overline{D})^m.
\end{equation*}
For more information, see the concluding remarks in the last section of this paper, where we discuss the fine structures on the $S$-spectrum that consist of the function spaces and the related functional calculi.

\section{Preliminaries on Clifford Algebras and Clifford modules}

Clifford algebras are associative algebras that generalize complex numbers and quaternions. They provide a mathematical framework to study operator theory in a noncommutative context. The fact that $n$-tuples, or more generally $2^n$-tuples, of noncommuting operators are embedded into a Clifford algebra, produces the benefit that for Clifford operators, the notion of spectrum turns out to be very clear and explicit, even if it differs from the classical notion in complex operator theory. Precisely, the real {\it Clifford algebra} $\mathbb{R}_n$ over $n$ {\it imaginary units} $e_1,\dots,e_n$, which satisfy the relations
\begin{equation*}
e_i^2=-1\qquad\text{and}\qquad e_ie_j=-e_je_i,\qquad i\neq j\in\{1,\dots,n\},
\end{equation*}
is given by
\begin{equation*}
\mathbb{R}_n:=\Big\{\sum\nolimits_{A\in\mathcal{A}}s_Ae_A\;\Big|\;s_A\in\mathbb{R},\,A\in\mathcal{A}\Big\}.
\end{equation*}
Here we used the index set
\begin{equation*}
\mathcal{A}:=\big\{(i_1,\dots,i_r)\;\big|\;r\in\{0,\dots,n\},\,1\leq i_1<\dots<i_r\leq n\big\},
\end{equation*}
and the {\it basis vectors} $e_A:=e_{i_1}\dots e_{i_r}$. Note that for $A=\emptyset$ the empty product of imaginary units is the real number $e_\emptyset=1$. Furthermore, we define for every Clifford number $s\in\mathbb{R}_n$ its {\it conjugate} and its {\it absolute value}
\begin{equation}\label{Eq_Conjugate_Norm}
\overline{s}:=\sum\nolimits_{A\in\mathcal{A}}(-1)^{\frac{|A|(|A|+1)}{2}}s_Ae_A\qquad\text{and}\qquad|s|^2:=\sum\nolimits_{A\in\mathcal{A}}s_A^2.
\end{equation}
An important subset of the Clifford numbers are the so called {\it paravectors}
\begin{equation*}
\mathbb{R}^{n+1}:=\big\{s_0+s_1e_1+\dots+s_ne_n\;\big|\;s_0,s_1,\dots,s_n\in\mathbb{R}\big\}.
\end{equation*}
For any paravector $s\in\mathbb{R}^{n+1}$, we define the {\it imaginary part}
\begin{equation*}
\Im(s):=s_1e_1+\dots+s_ne_n,
\end{equation*}
and the conjugate and the modulus in \eqref{Eq_Conjugate_Norm} reduce to
\begin{equation*}
\overline{s}=s_0-s_1e_1-\dots-s_ne_n\qquad\text{and}\qquad|s|^2=s_0^2+s_1^2+\dots+s_n^2.
\end{equation*}
The sphere of purely imaginary paravectors with modulus $1$, is defined by
\begin{equation*}
\mathbb{S}:=\big\{J\in\mathbb{R}^{n+1}\;\big|\;J_0=0,\,|J|=1\big\}.
\end{equation*}
Any element $J\in\mathbb{S}$ satisfies $J^2=-1$ and hence the corresponding hyperplane
\begin{equation*}
\mathbb{C}_J:=\big\{x+Jy\;\big|\;x,y\in\mathbb{R}\big\}
\end{equation*}
is an isomorphic copy of the complex numbers. Moreover, for every paravector $s\in\mathbb{R}^{n+1}$ we consider the corresponding {\it $(n-1)$--sphere}
\begin{equation*}
[s]:=\big\{s_0+J|\Im(s)|\;\big|\;J\in\mathbb{S}\big\}.
\end{equation*}
A subset $U\subseteq\mathbb{R}^{n+1}$ is called {\it axially symmetric}, if $[s]\subseteq U$ for every $s\in U$. \medskip

In the hypercomplex context, several notions of holomorphic functions exist, each with distinct series expansions. For example, the classical theory related to the Dirac operator leads to monogenic function, which admit expansions in terms of Fueter polynomials. In this article we consider slice hyperholomorphic functions, which allow power series expansions, and are particularly suitable for use with Clifford operators.

\begin{defi}[Slice hyperholomorphic functions]
Let $U\subseteq\mathbb{R}^{n+1}$ be open, axially symmetric and consider
\begin{equation*}
\mathcal{U}:=\big\{(x,y)\in\mathbb{R}^2\;\big|\;x+\mathbb{S}y\subseteq U\big\}.
\end{equation*}
A function $f:U\rightarrow\mathbb{R}_n$ is called {\it left} (resp. {\it right}) {\it slice hyperholomorphic}, if there exist continuously differentiable functions $f_0,f_1:\mathcal{U}\rightarrow\mathbb{R}_n$, such that for every $(x,y)\in\mathcal{U}$:

\begin{enumerate}
\item[i)] The function $f$ admits for every $J\in\mathbb{S}$ the representation
\begin{equation*}
f(x+Jy)=f_0(x,y)+Jf_1(x,y),\quad\Big(\text{resp.}\;f(x+Jy)=f_0(x,y)+f_1(x,y)J\Big).
\end{equation*}

\item[ii)] The functions $f_0,f_1$ satisfy the {\it compatibility conditions}
\begin{equation*}
f_0(x,-y)=f_0(x,y)\qquad\text{and}\qquad f_1(x,-y)=-f_1(x,y).
\end{equation*}

\item[iii)] The functions $f_0,f_1$ satisfy the {\it Cauchy-Riemann equations}
\begin{equation*}
\frac{\partial}{\partial x}f_0(x,y)=\frac{\partial}{\partial y}f_1(x,y)\qquad\text{and}\qquad\frac{\partial}{\partial y}f_0(x,y)=-\frac{\partial}{\partial x}f_1(x,y).
\end{equation*}
\end{enumerate}

The class of left (resp. right) slice hyperholomorphic functions on $U$ is denoted by $\mathcal{SH}_L(U)$ (resp. $\mathcal{SH}_R(U)$). In the special case that $f_0$ and $f_1$ are real valued, we call $f$ {\it intrinsic} and the space of intrinsic functions is denoted by $\mathcal{N}(U)$.
\end{defi}

Next, we turn our attention to Clifford modules over $\mathbb{R}_n$. For a real Banach space $V_\mathbb{R}$ with norm $\Vert\cdot\Vert_\mathbb{R}$, we define the corresponding \textit{Banach module}
\begin{equation}\label{Eq_Banach_module}
V:=\Big\{\sum\nolimits_{A\in\mathcal{A}}v_A\otimes e_A\;\Big|\;v_A\in V_\mathbb{R},\,A\in\mathcal{A}\Big\},
\end{equation}
and equip it with the \textit{norm}
\begin{equation*}
\Vert v\Vert^2:=\sum\nolimits_{A\in\mathcal{A}}\Vert v_A\Vert_\mathbb{R}^2,\qquad v\in V.
\end{equation*}
For any vector $v=\sum_{A\in\mathcal{A}}v_A\otimes e_A\in V$ and any Clifford number $s=\sum_{B\in\mathcal{A}}s_Be_B\in\mathbb{R}_n$, we establish the left and the right scalar multiplication
\begin{align*}
sv:=&\sum\nolimits_{A,B\in\mathcal{A}}(s_Bv_A)\otimes(e_Be_A),\qquad\textit{(left multiplication)} \\
vs:=&\sum\nolimits_{A,B\in\mathcal{A}}(v_As_B)\otimes(e_Ae_B).\hspace{0.84cm}\textit{(right multiplication)}
\end{align*}
We recall the well known properties of these products
\begin{align*}
\Vert sv\Vert&\leq 2^{\frac{n}{2}}|s|\Vert v\Vert\qquad\text{and}\qquad\Vert vs\Vert\leq 2^{\frac{n}{2}}|s|\Vert v\Vert,\qquad\text{if }s\in\mathbb{R}_n, \\
\Vert sv\Vert&=\Vert vs\Vert=|s|\Vert v\Vert,\hspace{4.83cm}\text{if }s\in\mathbb{R}^{n+1}.
\end{align*}
We will call $V$ a \textit{$\mathbb{R}$-reflexive Banach module}, if $V$ is reflexive with respect to its $\mathbb{R}$-linear dual space. For our particular space $V$ in \eqref{Eq_Banach_module} this means that $V$ is $\mathbb{R}$-reflexive, if and only if $V_\mathbb{R}$ is reflexive. \medskip

Next, let us denote the set of \textit{bounded right linear operators}
\begin{equation*}
\mathcal{B}(V):=\big\{T:V\rightarrow V\text{ right linear}\;\big|\;\dom(T)=V,\,T\text{ is bounded}\big\},
\end{equation*}
as well as the space of \textit{closed right linear operators}
\begin{equation*}
\mathcal{K}(V):=\big\{T:V\rightarrow V\text{ right linear}\;\big|\;\dom(T)\subseteq V\text{ is right linear},\,T\text{ is closed}\big\}.
\end{equation*}
If an unbounded operator $T$ is not closed, one may ask the question if there exists a closed extension. If this is the case we call the operator \textit{closable}, and the smallest closed extension $\overline{T}$ its \textit{closure}. \medskip

In the case an operator is not closable, which happens to be the case for the right $H^\infty$-calculus in \eqref{Eq_Hinfty}, the closure in some sense exceeds the notion of operators. We will now introduce the concept of \textit{multivalued right linear operators} (also known as right linear relations) to overcome this problem. The idea is based on the fact that the graph
\begin{equation*}
\text{graph}(T)=\big\{(v,Tv)\;\big|\;v\in\dom(T)\big\},
\end{equation*}
of a right linear operator $T$, is a right linear subspace of $V\times V$.

\begin{defi}[Multivalued operators]\label{defi_Multivalued_operators}
Any right linear subspace $\mathcal{T}\subseteq V\times V$ is called a \textit{right linear multivalued operator}.
\end{defi}

Let now $\mathcal{T},\mathcal{S}\subseteq V\times V$ be two right linear multivalued operators, and $s\in\mathbb{R}_n$. Then the sum and product of multivalued operators is defined as
\begin{align}
\mathcal{T}+\mathcal{S}:=&\big\{(v,w_1+w_2)\;\big|\;(v,w_1)\in\mathcal{T},\,(v,w_2)\in\mathcal{S}\big\}, \notag \\
s\mathcal{T}:=&\big\{(v,sw)\;\big|\;(v,w)\in\mathcal{T}\big\}, \label{Eq_Sum_product} \\
\mathcal{T}\circ\mathcal{S}:=&\big\{(v,u)\in V\times V\;\big|\;\exists w\in V\text{ such that }(v,w)\in\mathcal{S}\text{ and }(w,u)\in\mathcal{T}\big\}. \notag
\end{align}
We can also equip any multivalued operator naturally with a domain
\begin{equation}\label{Eq_Domain}
\dom(\mathcal{T}):=\big\{v\in V\;|\;\exists w\in V\text{ such that}(v,w)\in\mathcal{T}\big\}.
\end{equation}
The great advantage of treating multivalued operators in our case is, that every right linear multivalued operator $\mathcal{T}$ admits a closure
\begin{equation}\label{Eq_Closure}
\overline{\mathcal{T}}:=\big\{(v,w)\in V\times V\;\big|\;\exists(v_n,w_n)_n\in\mathcal{T}\text{ such that }(v,w)=\lim_{n\rightarrow\infty}(v_n,w_n)\big\},
\end{equation}
which is again a right linear multivalued operator. Note, that in the case that $\mathcal{T}$ and $\mathcal{S}$ are graphs of right linear operator $T$ and $S$, the definitions \eqref{Eq_Sum_product} coincide with the classical sum and product of operators. Also the closure \eqref{Eq_Closure} of multivalued operators coincides with the classical closure of operators in the case that $T$ is closable. A more detailed introduction into multivalued operators can for example be found in \cite[Section 1]{BHS20}.

\section{Properties of bisectorial operators}

We collect in this section important properties of bisectorial Clifford operators. We state and prove some consequences of the bisectoriality estimate of the left $S$-resolvent in Lemma~\ref{lem_Resolvent_estimates}. For reflexive Banach modules we show in Theorem~\ref{thm_Domain_dense} that bisectorial operators are always densely defined. Finally, in Theorem~\ref{thm_dom_ran_density} we study also powers of a bisectorial operator. \medskip

In order to introduce bisectorial operators we define for every $\omega\in(0,\frac{\pi}{2})$ the \textit{double sector} \medskip

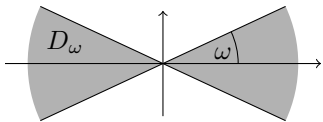
\begin{minipage}{0.34\textwidth}
\begin{center}
\begin{tikzpicture}
\fill[black!30] (0,0)--(1.63,0.76) arc (25:-25:1.8)--(0,0)--(-1.63,-0.76) arc (205:155:1.8);
\draw (-1.63,0.76)--(1.63,-0.76);
\draw (-1.63,-0.76)--(1.63,0.76);
\draw[->] (-2.1,0)--(2.1,0);
\draw[->] (0,-0.7)--(0,0.7);
\draw (1,0) arc (0:25:1) (0.8,-0.07) node[anchor=south] {\small{$\omega$}};
\draw (-1.3,0) node[anchor=south] {\small{$D_\omega$}};
\end{tikzpicture}
\end{center}
\end{minipage}
\begin{minipage}{0.65\textwidth}
\begin{equation}\label{Eq_Domega}
D_\omega:=\big\{re^{J\phi}\;\big|\;r>0,\,J\in\mathbb{S},\,\phi\in I_\omega\big\},
\end{equation}
using the union of intervals
\begin{equation*}
I_\omega:=(-\omega,\omega)\cup(\pi-\omega,\pi+\omega).
\end{equation*}
\end{minipage}

\begin{defi}[Bisectorial operators]
An operator $T\in\mathcal{K}(V)$ is called \textit{bisectorial of angle} $\omega\in(0,\frac{\pi}{2})$, if its $S$-spectrum is contained in the closed double sector
\begin{equation*}
\sigma_S(T)\subseteq\overline{D_\omega},
\end{equation*}
and for every $\varphi\in(\omega,\frac{\pi}{2})$ there exists $C_\varphi\geq 0$, such that the left $S$-resolvent \eqref{Eq_SL_SR} satisfies
\begin{equation}\label{Eq_SL_estimate}
\Vert S_L^{-1}(s,T)\Vert\leq\frac{C_\varphi}{|s|},\qquad s\in\mathbb{R}^{n+1}\setminus(D_\varphi\cup\{0\}).
\end{equation}
\end{defi}

We will now show that the estimate \eqref{Eq_SL_estimate} on the left $S$-resolvent implies similar estimates on the right $S$-resolvent $S_R^{-1}(s,T)$ and also on the pseudo resolvent operator $Q_s[T]^{-1}$.

\begin{lem}\label{lem_Resolvent_estimates}
Let $T\in\mathcal{K}(V)$ be bisectorial of angle $\omega\in(0,\frac{\pi}{2})$. Then for every $\varphi\in(\omega,\frac{\pi}{2})$, and with the corresponding constant $C_\varphi$ from \eqref{Eq_SL_estimate}, there holds the norm estimates

\begin{enumerate}
\item[i)] $\Vert S_R^{-1}(s,T)\Vert\leq\frac{2C_\varphi}{|s|}$,\hspace{2.7cm} $s\in\mathbb{R}^{n+1}\setminus(D_\varphi\cup\{0\})$; \medskip

\item[ii)] $\Vert Q_s[T]^{-1}\Vert\leq\frac{2C_\varphi^2}{|s|^2}$,\hspace{2.94cm} $s\in\mathbb{R}^{n+1}\setminus(D_\varphi\cup\{0\})$; \medskip

\item[iii)] $\Vert TQ_s[T]^{-1}\Vert\leq\frac{2C_\varphi^2+C_\varphi}{|s|}$,\hspace{2.03cm} $s\in\mathbb{R}^{n+1}\setminus(D_\varphi\cup\{0\})$, \medskip

\item[iv)] $\Vert T^2Q_s[T]^{-1}\Vert\leq 1+2C_\varphi+2C_\varphi^2$,\qquad $s\in\mathbb{R}^{n+1}\setminus(D_\varphi\cup\{0\})$.
\end{enumerate}
\end{lem}

\begin{proof}
i)\;\;Let us choose $J\in\mathbb{S}$, such that $s\in\mathbb{C}_J$. Then it is straight forward to check, by plugging in the definitions \eqref{Eq_SL_SR} of the $S$-resolvent operators, that
\begin{subequations}
\begin{align}
S_R^{-1}(s,T)+S_R^{-1}(\overline{s},T)&=2(s_0-T)Q_s[T]^{-1}=S_L^{-1}(s,T)+S_L^{-1}(\overline{s},T), \label{Eq_Resolvent_estimates_1} \\
S_R^{-1}(s,T)-S_R^{-1}(\overline{s},T)&=-2\Im(s)Q_s[T]^{-1}=-J\big(S_L^{-1}(s,T)-S_L^{-1}(\overline{s},T)\big)J. \label{Eq_Resolvent_estimates_2}
\end{align}
\end{subequations}
By adding these two equations, we can express the right $S$-resolvent in terms of the left $S$-resolvent in the way
\begin{equation*}
S_R^{-1}(s,T)=\frac{1}{2}\big(S_L^{-1}(s,T)+S_L^{-1}(\overline{s},T)\big)-\frac{J}{2}\big(S_L^{-1}(s,T)-S_L^{-1}(\overline{s},T)\big)J.
\end{equation*}
Using the given estimate \eqref{Eq_SL_estimate} of the left $S$-resolvent, we can then estimate
\begin{equation*}
\Vert S_R^{-1}(s,T)\Vert\leq\frac{1}{2}\Big(\frac{C_\varphi}{|s|}+\frac{C_\varphi}{|\overline{s}|}\Big)+\frac{1}{2}\Big(\frac{C_\varphi}{|s|}+\frac{C_\varphi}{|\overline{s}|}\Big)=\frac{2C_\varphi}{|s|}.
\end{equation*}
ii)\;\;From the identity \eqref{Eq_Resolvent_estimates_1}, we get
\begin{equation}\label{Eq_Resolvent_estimates_3}
\big(S_L^{-1}(s,T)+S_L^{-1}(\overline{s},T)\big)^2=4(s_0-T)^2Q_s[T]^{-2},
\end{equation}
and from \eqref{Eq_Resolvent_estimates_2}, we obtain
\begin{equation}\label{Eq_Resolvent_estimates_4}
\big(S_L^{-1}(s,T)-S_L^{-1}(\overline{s},T)\big)J\big(S_L^{-1}(s,T)-S_L^{-1}(\overline{s},T)\big)J=-4\Im(s)^2Q_s[T]^{-2}.
\end{equation}
Adding now the equations \eqref{Eq_Resolvent_estimates_3} and \eqref{Eq_Resolvent_estimates_4} allows us to write $Q_s[T]^{-1}$ in terms of the left $S$-resolvent in the form
\begin{align*}
\big(S_L^{-1}(s,T)+S_L^{-1}(\overline{s},T)\big)^2&+\big(S_L^{-1}(s,T)-S_L^{-1}(\overline{s},T)\big)J\big(S_L^{-1}(s,T)-S_L^{-1}(\overline{s},T)\big)J \\
&=4(s_0-T)^2Q_s[T]^{-2}-4\Im(s)^2Q_s[T]^{-2} \\
&=4(T^2-2s_0T+|s|^2)Q_s[T]^{-2}=4Q_s[T]^{-1}.
\end{align*}
Consequently, we can estimate
\begin{equation*}
\Vert Q_s[T]^{-1}\Vert\leq\frac{1}{4}\Big(\frac{C_\varphi}{|s|}+\frac{C_\varphi}{|s|}\Big)^2+\frac{1}{4}\Big(\frac{C_\varphi}{|s|}+\frac{C_\varphi}{|s|}\Big)\Big(\frac{C_\varphi}{|s|}+\frac{C_\varphi}{|s|}\Big)=\frac{2C_\varphi^2}{|s|^2}.
\end{equation*}
iii)\;\;Rewriting the definition \eqref{Eq_SL_SR} as $TQ_s[T]^{-1}=Q_s[T]^{-1}\overline{s}-S_L^{-1}(s,T)$, we can estimate
\begin{equation*}
\Vert TQ_s[T]^{-1}\Vert\leq\Vert Q_s[T]^{-1}\overline{s}\Vert+\Vert S_L^{-1}(s,T)\Vert\leq\frac{2C_\varphi^2}{|s|}+\frac{C_\varphi}{|s|}=\frac{2C_\varphi^2+C_\varphi}{|s|}.
\end{equation*}
iv)\;\;Finally, since we can write $T^2Q_s[T]^{-1}=1-S_L^{-1}(s,T)s+TQ_s[T]^{-1}\overline{s}$, it follows immediately from \eqref{Eq_SL_estimate} and iii), that
\begin{equation*}
\Vert T^2Q_s[T]^{-1}\Vert\leq 1+|s|\Vert S_L^{-1}(s,T)\Vert+|s|\Vert TQ_s[T]^{-1}\Vert\leq 1+2C_\varphi+2C_\varphi^2. \qedhere
\end{equation*}
\end{proof}

The following theorem will be crucial in Section~\ref{sec_Right_Hinfty} for the well definedness of the right $H^\infty$-functional calculus. It states that bisectorial operators are on the one hand densely defined, but more importantly that the Banach module $V$ decomposes into the kernel and the closure of the range of $T$.

\begin{thm}\label{thm_Domain_dense}
Let $V$ be a $\mathbb{R}$-reflexive Banach module and
$T\in\mathcal{K}(V)$ bisectorial of angle $\omega\in(0,\frac{\pi}{2})$. Then

\begin{enumerate}
\item[i)] $V=\overline{\dom}(T)$; \medskip

\item[ii)] $V=\ker(T)+\overline{\ran}(T)$\qquad and \qquad $\ker(T)\cap\overline{\ran}(T)=\emptyset$.
\end{enumerate}
\end{thm}

\begin{proof}
i)\;\;Let $v\in V$ and fix any $\varphi\in(\omega,\frac{\pi}{2})$. Then consider the family of vectors
\begin{equation*}
v_s:=S_L^{-1}(s,T)sv,\qquad s\in\mathbb{R}^{n+1}\setminus(D_\varphi\cup\{0\}).
\end{equation*}
Due to \eqref{Eq_SL_estimate}, these vectors admit the uniform upper bound
\begin{equation*}
\Vert v_s\Vert=\Vert S_L^{-1}(s,T)sv\Vert\leq\frac{C_\varphi}{|s|}|s|\Vert v\Vert=C_\varphi\Vert v\Vert,\qquad s\in\mathbb{R}^{n+1}\setminus(D_\varphi\cup\{0\}).
\end{equation*}
Since $V$ is $\mathbb{R}$-reflexive, there exists $(s_n)_n\in\mathbb{R}^{n+1}\setminus(D_\varphi\cup\{0\})$ and some $w\in V$, such that
\begin{equation}\label{Eq_ker_ran_decomposition_1}
\lim\limits_{n\rightarrow\infty}|s_n|=\infty\qquad\text{and}\qquad\lim\limits_{n\rightarrow\infty}v_{s_n}=w,\quad\text{weakly in }V.
\end{equation}
Note, that the weak convergence is understood with respect to the dual space of $\mathbb{R}$-linear functionals. From the second limit in \eqref{Eq_ker_ran_decomposition_1}, and the left $S$-resolvent identity
\begin{equation}\label{Eq_ker_ran_decomposition_6}
TS_L^{-1}(s_n,T)=S_L^{-1}(s_n,T)s_n-1,
\end{equation}
which follows immediately from its definition \eqref{Eq_SL_SR}, it follows that
\begin{equation}\label{Eq_ker_ran_decomposition_2}
\lim\limits_{n\rightarrow\infty}TS_L^{-1}(s_n,T)v=\lim\limits_{n\rightarrow\infty}S_L^{-1}(s_n,T)s_nv-v=w-v,\quad\text{weakly in }V.
\end{equation}
Moreover, from the first limit in \eqref{Eq_ker_ran_decomposition_1} and the estimate \eqref{Eq_SL_estimate}, we also obtain the convergence
\begin{equation}\label{Eq_ker_ran_decomposition_3}
\lim\limits_{n\rightarrow\infty}\Vert S_L^{-1}(s_n,T)v\Vert\leq\lim\limits_{n\rightarrow\infty}\frac{C_\varphi}{|s_n|}\Vert v\Vert=0.
\end{equation}
Since $T$ is a closed operator, it is also weakly closed, which means that from \eqref{Eq_ker_ran_decomposition_2} and \eqref{Eq_ker_ran_decomposition_3} one concludes $w-v=T0=0$. Hence, the weak convergence \eqref{Eq_ker_ran_decomposition_1} turns into the weak convergence $\lim_{n\rightarrow\infty}v_{s_n}=v$, i.e. $v$ is weakly approximated by elements $v_{s_n}\in\dom(T)$. Hence $v$ is contained in the weak closure of $\dom(T)$, which is the same as being in the closure $v\in\overline{\dom}(T)$. \medskip

ii)\;\; We take the same family $v_s$ of vectors from \eqref{Eq_ker_ran_decomposition_3}, and analogously to \eqref{Eq_ker_ran_decomposition_1} we choose this time a sequence $(s_n)_n\in\mathbb{R}^{n+1}\setminus(D_\varphi\cup\{0\})$ and some $w\in V$, such that
\begin{equation}\label{Eq_ker_ran_decomposition_5}
\lim\limits_{n\rightarrow\infty}|s_n|=0\qquad\text{and}\qquad\lim\limits_{n\rightarrow\infty}v_{s_n}=w,\quad\text{weakly in }V.
\end{equation}
Using once more the identity \eqref{Eq_ker_ran_decomposition_6} as well as the estimate \eqref{Eq_SL_estimate}, there also converges
\begin{align}
\lim\limits_{n\rightarrow\infty}\Vert Tv_{s_n}\Vert&=\lim\limits_{n\rightarrow\infty}\Vert TS_L^{-1}(s_n,T)s_nv\Vert=\lim\limits_{n\rightarrow\infty}\Vert(S_L^{-1}(s_n,T)s_n-1)s_nv\Vert \notag \\
&\leq\lim\limits_{n\rightarrow\infty}\big(\Vert S_L^{-1}(s_n,T)\Vert|s_n|+1\big)|s_n|\Vert v\Vert\leq(C_\varphi+1)\lim\limits_{n\rightarrow\infty}|s_n|\Vert v\Vert=0.\label{Eq_ker_ran_decomposition_4}
\end{align}
As above, $T$ is closed, hence weakly closed, and we conclude from the two limits \eqref{Eq_ker_ran_decomposition_5} and \eqref{Eq_ker_ran_decomposition_4}, that $w\in\dom(T)$ and $Tw=0$, i.e. $w\in\ker(T)$. Moreover, with the $S$-resolvent identity \eqref{Eq_ker_ran_decomposition_6}, there converges
\begin{equation*}
\lim\limits_{n\rightarrow\infty}TS_L^{-1}(s_n,T)v=\lim\limits_{n\rightarrow\infty}S_L^{-1}(s_n,T)s_nv-v=w-v,\quad\text{weakly in }V.
\end{equation*}
This shows that $w-v$ is contained in the weak closure of $\ran(T)$, which is the same as $w-v\in\overline{\ran}(T)$. Altogether we have decomposed $v$ into
\begin{equation*}
v=w+(v-w)\in\ker(T)+\overline{\ran}(T).
\end{equation*}
In order to show that the sum is direct, let $v\in\ker(T)\cap\overline{\ran}(T)$. Then, from the first limit in the upcoming Lemma~\ref{lem_Properties_dom_ran}~ii), we obtain
\begin{equation*}
v=\lim_{\substack{s\rightarrow 0 \\ s\in\mathbb{R}^{n+1}\setminus D_\varphi}}(T-\overline{s})TQ_s[T]^{-1}v=\lim_{\substack{s\rightarrow 0 \\ s\in\mathbb{R}^{n+1}\setminus D_\varphi}}(T-\overline{s})Q_s[T]^{-1}0=0,
\end{equation*}
where we used that $TQ_s[T]^{-1}v=Q_s[T]^{-1}Tv$ commutes since $v\in\ker(T)\subseteq\dom(T)$.
\end{proof}

\begin{lem}\label{lem_Properties_dom_ran}
Let $V$ be a $\mathbb{R}$-reflexive Banach module and
$T\in\mathcal{K}(V)$ bisectorial of angle $\omega\in(0,\frac{\pi}{2})$. Then for every fixed $\varphi\in(\omega,\frac{\pi}{2})$ there exist the following limits: \medskip

\begin{enumerate}
\item[i)] $\lim\limits_{\substack{s\rightarrow\infty \\ s\in\mathbb{R}^{n+1}\setminus D_\varphi}}T^2Q_s[T]^{-1}v=0$,\hspace{0.75cm} $v\in V$;

\item[ii)] $\lim\limits_{\substack{s\rightarrow 0 \\ s\in\mathbb{R}^{n+1}\setminus D_\varphi}}|s|^2Q_s[T]^{-1}v=0$,\qquad $v\in\overline{\ran}(T)$.
\end{enumerate}
\end{lem}

\begin{proof}
i)\;\;Let $v\in V$. By Theorem~\ref{thm_Domain_dense}~i), there exists a sequence $(v_n)_n\in\dom(T)$ with
\begin{equation}\label{Eq_Properties_dom_ran_1}
\lim\limits_{n\rightarrow\infty}\Vert v-v_n\Vert=0.
\end{equation}
Since there commutes $TQ_s[T]^{-1}v_n=Q_s[T]^{-1}Tv_n$ because of $v_n\in\dom(T)$, we can use the inequalities in Lemma~\ref{lem_Resolvent_estimates}~iii)~\&~iv), to estimate
\begin{align*}
\Vert T^2Q_s[T]^{-1}v\Vert&\leq\Vert T^2Q_s[T]^{-1}(v-v_n)\Vert+\Vert TQ_s[T]^{-1}Tv_n\Vert \\
&\leq(1+2C_\varphi+2C_\varphi^2)\Vert v-v_n\Vert+\frac{2C_\varphi^2+C_\varphi}{|s|}\Vert Tv_n\Vert.
\end{align*}
Sending $s\rightarrow\infty$, the second term of the right hand side vanishes, and the estimate reduces to
\begin{equation*}
\limsup_{\substack{s\rightarrow\infty \\ s\in\mathbb{R}^{n+1}\setminus D_\varphi}}\Vert T^2Q_s[T]^{-1}v\Vert\leq(1+2C_\varphi+2C_\varphi^2)\Vert v-v_n\Vert.
\end{equation*}
However, since this is true for every $n\in\mathbb{N}$, we can also take the limit $n\rightarrow\infty$ of the right hand side and with \eqref{Eq_Properties_dom_ran_1} we end up with the stated convergence
\begin{equation*}
\lim_{\substack{s\rightarrow\infty \\ s\in\mathbb{R}^{n+1}\setminus D_\varphi}}\Vert T^2Q_s[T]^{-1}v\Vert=0.
\end{equation*}
ii)\;\;Let $v\in\overline{\ran}(T)$, i.e. there exists a sequence $(u_n)_n\in\dom(T)$, such that
\begin{equation}\label{Eq_Properties_dom_ran_2}
\lim\limits_{n\rightarrow\infty}\Vert v-Tu_n\Vert=0.
\end{equation}
We can now use the inequalities in Lemma~\ref{lem_Resolvent_estimates}~ii)~\&~iii), to estimate
\begin{align*}
\Vert|s|^2Q_s[T]^{-1}v\Vert&\leq\Vert|s|^2Q_s[T]^{-1}(v-Tu_n)\Vert+\Vert|s|^2TQ_s[T]^{-1}u_n\Vert \\
&\leq 2C_\varphi^2\Vert v-Tu_n\Vert+(2C_\varphi^2+C_\varphi)|s|\Vert u_n\Vert.
\end{align*}
Sending $s\rightarrow 0$, the second term on the right hand side vanishes and the estimate reduces to
\begin{equation*}
\limsup\limits_{\substack{s\rightarrow 0 \\ s\in\mathbb{R}^{n+1}\setminus D_\varphi}}\Vert|s|^2Q_s[T]^{-1}v\Vert\leq 2C_\varphi^2\Vert v-Tu_n\Vert.
\end{equation*}
However, since this is true for every $n\in\mathbb{N}$, we can also take the limit $n\rightarrow\infty$ of the right hand side, and with \eqref{Eq_Properties_dom_ran_2} we end up with the stated convergence
\begin{equation*}
\lim\limits_{\substack{s\rightarrow 0 \\ s\in\mathbb{R}^{n+1}\setminus D_\varphi}}\Vert|s|^2Q_s[T]^{-1}v\Vert=0. \qedhere
\end{equation*}
\end{proof}

The next theorem describes properties of the powers of the bisectorial operator $T$.

\begin{thm}\label{thm_dom_ran_density}
Let $V$ be a $\mathbb{R}$-reflexive Banach module and $T\in\mathcal{K}(V)$ bisectorial of angle $\omega\in(0,\frac{\pi}{2})$. Then for every $m\in\mathbb{N}$, there is
\begin{equation}\label{Eq_Tm_ran_dense}
\overline{\dom(T^m)\cap\ran(T^m)}=\overline{\ran}(T).
\end{equation}
If $T$ is also injective, then even
\begin{equation}\label{Eq_Tm_dense}
\overline{\dom(T^m)\cap\ran(T^m)}=V.
\end{equation}
\end{thm}

\begin{proof}
The inclusion $\text{\grqq}\subseteq\text{\grqq}$ in \eqref{Eq_Tm_ran_dense} is trivial. For the inclusion $\text{\grqq}\supseteq\text{\grqq}$, let us consider for every $n\in\mathbb{N}$ the operator
\begin{equation}\label{Eq_dom_ran_density_2}
r_n[T]:=n^2T^2(T^2+n^2)^{-1}\Big(T^2+\frac{1}{n^2}\Big)^{-1}.
\end{equation}
Note that since we can write $T^2+n^2=Q_{Jn}[T]$ and $T^2+\frac{1}{n^2}=Q_{\frac{J}{n}}[T]$, for some arbitrary $J\in\mathbb{S}\subseteq\rho_S(T)$, the operator $r_n[T]$ is well defined and bounded. \medskip

In the \textit{first step} we want to show that for every $m\in\mathbb{N}$, there converges
\begin{equation}\label{Eq_dom_ran_density_5}
\lim\limits_{n\rightarrow\infty}r_n[T]^mu=u,\qquad u\in\overline{\ran}(T).
\end{equation}
For the induction start $m=1$, we rewrite the operator \eqref{Eq_dom_ran_density_2} as
\begin{align*}
r_n[T]&=1-T^2(T^2+n^2)^{-1}-(T^2+n^2)^{-1}\Big(T^2+\frac{1}{n^2}\Big)^{-1} \\
&=1-T^2Q_{Jn}[T]^{-1}-Q_{Jn}[T]^{-1}Q_{\frac{J}{n}}[T]^{-1}.
\end{align*}
Using the bounds of $Q_{Jn}[T]^{-1}$ in Lemma~\ref{lem_Resolvent_estimates}~ii), we can estimate
\begin{align*}
\Vert r_n[T]u-u\Vert&\leq\Vert T^2Q_{Jn}[T]^{-1}u\Vert+\Vert Q_{Jn}[T]^{-1}\Vert\Vert Q_{\frac{J}{n}}[T]^{-1}u\Vert \\
&\leq\Vert T^2Q_{Jn}[T]^{-1}u\Vert+\frac{2C_\varphi^2}{n^2}\Vert Q_{\frac{J}{n}}[T]^{-1}\Vert,\qquad u\in\overline{\ran}(T).
\end{align*}
Using the limits in Lemma~\ref{lem_Properties_dom_ran}~i)~\&~ii), we then obtain
\begin{equation*}
\lim\limits_{n\rightarrow\infty}\Vert r_n[T]u-u\Vert\leq\lim\limits_{n\rightarrow\infty}\Vert T^2Q_{Jn}[T]^{-1}u\Vert+2C_\varphi^2\lim\limits_{n\rightarrow\infty}\Big\Vert\frac{1}{n}Q_{\frac{J}{n}}[T]^{-1}u\Big\Vert=0.
\end{equation*}
For the induction step $m\rightarrow m+1$, we estimate
\begin{align*}
\Vert r_n[T]^{m+1}u-u\Vert&\leq\Vert r_n[T]\Vert\Vert r_n[T]^mu-u\Vert+\Vert r_n[T]u-u\Vert \\
&\leq 2C_\varphi^2(1+2C_\varphi+2C_\varphi^2)\Vert r_n[T]^mu-u\Vert+\Vert r_n[T]u-u\Vert.
\end{align*}
where in the second line we used estimate
\begin{equation*}
\Vert r_n[T]\Vert=n^2\big\Vert Q_{Jn}[T]^{-1}T^2Q_{\frac{J}{n}}[T]^{-1}\big\Vert\leq 2C_\varphi^2(1+2C_\varphi+2C_\varphi^2),
\end{equation*}
which follows from Lemma~\ref{lem_Resolvent_estimates}~ii)~\&~iv). Hence, using the induction assumption $m$ as well as the induction start $m=1$, we conclude the converges
\begin{equation*}
\lim\limits_{n\rightarrow\infty}\Vert r_n[T]^{m+1}u-u\Vert\leq 2C_\varphi^2(1+2C_\varphi+2C_\varphi^2)\lim\limits_{n\rightarrow\infty}\Vert r_n[T]^mu-u\Vert+\lim\limits_{n\rightarrow\infty}\Vert r_n[T]u-u\Vert=0.
\end{equation*}
Hence we have proven \eqref{Eq_dom_ran_density_5}. Next, since we can write the $m$-th power of $r_n[T]$ as
\begin{equation*}
r_n[T]^m=n^mT^{2m}(T^2+n)^{-m}\Big(T^2+\frac{1}{n}\Big)^{-m},
\end{equation*}
it is obvious that
\begin{equation}\label{Eq_dom_ran_density_1}
\ran(r_n[T]^m)=\dom(T^{2m})\cap\ran(T^{2m}).
\end{equation}
Hence, for every $v\in\overline{\ran}(T)$ there is $v_n:=r_n[T]^mv\in\dom(T^{2m})\cap\ran(T^{2m})$, and we have already proven in \eqref{Eq_dom_ran_density_5} that there converges
\begin{equation*}
v=\lim\limits_{n\rightarrow\infty}v_n.
\end{equation*}
This proves that $v\in\overline{\dom(T^{2m})\cap\ran(T^{2m})}$, and we have verified \eqref{Eq_Tm_ran_dense} for all even powers of $T$. However, since the sets $\dom(T^m)\cap\ran(T^m)$ are getting smaller with increasing $m$, the equality \eqref{Eq_Tm_ran_dense} is then also true for every value $m\in\mathbb{N}$. \medskip

If $T$ is also injective, there is $\overline{\ran}(T)=V$ by Theorem~\ref{thm_Domain_dense}~ii), and \eqref{Eq_Tm_dense} follows from \eqref{Eq_Tm_ran_dense}.
\end{proof}

\section{The $H^\infty$-functional calculus right slice hyperholomorphic functions}\label{sec_Right_Hinfty}

This section contains the main results of this paper. We will start in Definition~\ref{defi_Omega_functional_calculus} by recalling the well established $\omega$-functional calculus for left and right holomorphic functions. We will do this in accordance with \cite[Definition 3.5]{MS24} and hence need the following spaces of slice hyperholomorphic functions.

\begin{defi}
For every $\theta\in(0,\frac{\pi}{2})$ let $D_\theta$ be the double sector from \eqref{Eq_Domega}. Then we define the function spaces
\begin{align*}
\text{i)}\;\;&\mathcal{SH}_L^0(D_\theta):=\bigg\{f\in\mathcal{SH}_L(D_\theta)\;\bigg|\;f\text{ is bounded},\;\int_0^\infty|f(re^{J\phi})|\frac{dr}{r}<\infty,\;J\in\mathbb{S},\phi\in I_\theta\bigg\}, \\
\text{ii)}\;\;&\mathcal{SH}_R^0(D_\theta):=\bigg\{f\in\mathcal{SH}_R(D_\theta)\;\bigg|\;f\text{ is bounded},\;\int_0^\infty|f(re^{J\phi})|\frac{dr}{r}<\infty,\;J\in\mathbb{S},\phi\in I_\theta\bigg\}, \\
\text{iii)}\;\;&\mathcal{N}^0(D_\theta):=\bigg\{f\in\mathcal{N}(D_\theta)\;\bigg|\;f\text{ is bounded},\;\int_0^\infty|f(re^{J\phi})|\frac{dr}{r}<\infty,\;J\in\mathbb{S},\phi\in I_\theta\bigg\}.
\end{align*}
\end{defi}

For those functions we then define the so called $\omega$-functional calculus.

\begin{defi}[$\omega$-functional calculus]\label{defi_Omega_functional_calculus}
Let $T\in\mathcal{K}(V)$ be bisectorial of angle $\omega\in(0,\frac{\pi}{2})$. Then for every $f\in\mathcal{SH}_L^0(D_\theta)$ (resp. $f\in\mathcal{SH}_R^0(D_\theta)$) we define the \textit{$\omega$-functional calculus}
\begin{align}
f(T):=&\frac{1}{2\pi}\int_{\partial D_\varphi\cap\mathbb{C}_J}S_L^{-1}(s,T)ds_Jf(s), \label{Eq_omega_functional_calculus} \\
\bigg(\text{resp.}\;\;f(T):=&\frac{1}{2\pi}\int_{\partial D_\varphi\cap\mathbb{C}_J}f(s)ds_JS_R^{-1}(s,T)\bigg), \label{Eq_Right_omega}
\end{align}
where $\varphi\in(\omega,\theta)$ and $J\in\mathbb{S}$ are arbitrary and the integral \eqref{Eq_omega_functional_calculus} is independent of the choice.
\end{defi}

As it is already worked out already in greater detail in \cite[Section 5]{MS24}, the $\omega$-functional calculus \eqref{Eq_omega_functional_calculus} can now be generalized to a larger class of left holomorphic functions. Since in this article we consider injective, sectorial operators $T$, we will define the $H^\infty$-functional calculus for the following natural class of polynomially growing functions, see also \cite[Proposition 5.2]{MS24}.

\begin{defi}\label{defi_SH_poly}
For every $\theta\in(0,\frac{\pi}{2})$ let $D_\theta$ be the double sector from \eqref{Eq_Domega}. Then we define the function spaces
\begin{align*}
\text{i)}\;\;&\mathcal{SH}_L^\poly(D_\theta):=\bigg\{f\in\mathcal{SH}_L(D_\theta)\;\bigg|\;\exists C,\alpha\geq 0: |f(s)|\leq C\Big(|s|^\alpha+\frac{1}{|s|^\alpha}\Big),\text{ for every }s\in D_\theta\bigg\}; \\
\text{ii)}\;\;&\mathcal{SH}_R^\poly(D_\theta):=\bigg\{f\in\mathcal{SH}_R(D_\theta)\;\bigg|\;\exists C,\alpha\geq 0: |f(s)|\leq C\Big(|s|^\alpha+\frac{1}{|s|^\alpha}\Big),\text{ for every }s\in D_\theta\bigg\}; \\
\text{iii)}\;\;&\mathcal{N}^\poly(D_\theta):=\bigg\{f\in\mathcal{N}(D_\theta)\;\bigg|\;\exists C,\alpha\geq 0: |f(s)|\leq C\Big(|s|^\alpha+\frac{1}{|s|^\alpha}\Big),\text{ for every }s\in D_\theta\bigg\}.
\end{align*}
\end{defi}

Using the following regularization procedure, the so called $H^\infty$-functional calculus generalizes the $\omega$-functional calculus at least for left slice hyperholomorphic functions $f\in\mathcal{SH}_L^\poly(D_\theta)$, see \cite[Definition 5.3]{MS24}.

\begin{defi}[Left $H^\infty$-functional calculus]\label{defi_Left_Hinfty}
Let $T\in\mathcal{K}(V)$ be sectorial of angle $\omega\in(0,\frac{\pi}{2})$ and $f\in\mathcal{SH}_L^\poly(D_\theta)$. then we define the \textit{left $H^\infty$-functional calculus}
\begin{equation}\label{Eq_Left_Hinfty}
f(T):=e(T)^{-1}(ef)(T),
\end{equation}
using the function $e(s)=\frac{s^m}{(1+s^2)^m}$, for some $m\in\mathbb{N}$ with $m>\alpha$ and $\alpha$ from Definition~\ref{defi_SH_poly}.
\end{defi}

The main purpose of this section is to extend this left $H^\infty$-functional calculus also to right slice hyperholomorphic functions $f\in\mathcal{SH}_R^\poly(D_\theta)$. Let us start with the following theorem, which gives a different representation of \eqref{Eq_Left_Hinfty} for intrinsic functions.

\begin{thm}[Reformulation of the intrinsic $H^\infty$-functional calculus]\label{thm_fT_as_closure}
Let $V$ be a $\mathbb{R}$-reflexive Banach module and $T\in\mathcal{K}(V)$ injective, bisectorial of angle $\omega\in(0,\frac{\pi}{2})$. Then for every $f\in\mathcal{N}^\poly(D_\theta)$, $\theta\in(\omega,\frac{\pi}{2})$, the operator $(fe)(T)e(T)^{-1}$ is closable, and the left $H^\infty$-functional calculus \eqref{Eq_Left_Hinfty} can be written as
\begin{equation}\label{Eq_fT_as_closure}
f(T)=\overline{(fe)(T)e(T)^{-1}}.
\end{equation}
Here $e(s)=\frac{s^m}{(1+s^2)^m}$, for some $m\in\mathbb{N}$ with $m>\alpha$ and $\alpha$ from Definition~\ref{defi_SH_poly}.
\end{thm}

\begin{proof}
For the inclusion $\text{\grqq}\supseteq\text{\grqq}$, we use the commutativity property \cite[Corollary~3.18~iii)]{MS24} of $e(T)$ and $(ef)(T)$, to obtain the operator inclusion
\begin{align}
(fe)(T)e(T)^{-1}&=e(T)^{-1}e(T)(fe)(T)e(T)^{-1} \notag \\
&=e(T)^{-1}(ef)(T)e(T)e(T)^{-1}\subseteq e(T)^{-1}(ef)(T)=f(T). \label{Eq_fT_as_closure_4}
\end{align}
Since the left $H^\infty$-functional calculus $f(T)$ is a closed operator, see \cite[Lemma 5.5]{MS24}, we conclude that $(fe)(T)e(T)^{-1}$ is closable, and
\begin{equation*}
\overline{(fe)(T)e(T)^{-1}}\subseteq f(T).
\end{equation*}
For the inverse inclusion $\text{\grqq}\subseteq\text{\grqq}$, let us consider the operators $r_n(T)$ from \eqref{Eq_dom_ran_density_2}. For any $v\in\dom(f(T))$ and the fixed value $m$ in the statement of the theorem, we consider the sequence of vectors
\begin{equation*}
v_n:=r_n[T]^mv,\qquad n\in\mathbb{N}.
\end{equation*}
Since $\overline{\ran}(T)=V$ by Theorem~\ref{thm_Domain_dense}~ii) and the injectivity of $T$, it is already proven in \eqref{Eq_dom_ran_density_5} that there converges
\begin{equation}\label{Eq_fT_as_closure_3}
\lim\limits_{n\rightarrow\infty}v_n=\lim\limits_{n\rightarrow\infty}r_n[T]^mv=v.
\end{equation}
By \cite[Theorem 3.16]{MS24}, the $\omega$-functional calculus of the regularizer function $e(s)=\frac{s^m}{(1+s^2)^m}$ can be written as
\begin{equation*}
e(T)=T^m(1+T^2)^{-m}.
\end{equation*}
Consequently, its range is given by
\begin{equation*}
\ran(e(T))=\dom(T^m)\cap\ran(T^m).
\end{equation*}
However, we also know the range of $r_n[T]^m$ from \eqref{Eq_dom_ran_density_1}, which then clearly gives us
\begin{equation}\label{Eq_fT_as_closure_1}
v_n=r_n[T]^mv\in\ran(r_n[T]^m)\subseteq\ran(e(T))=\dom\big((fe)(T)e(T)^{-1}\big).
\end{equation}
Using once more the operator inclusion \eqref{Eq_fT_as_closure_4} as well as the commutation
\begin{equation*}
f(T)r_n[T]^m\supseteq r_n[T]^mf(T),
\end{equation*}
from \cite[Corollary~5.11~i)]{MS24}, we can rewrite
\begin{equation*}
(fe)(T)e(T)^{-1}v_n=f(T)v_n=f(T)r_n[T]^mv=r_n[T]^mf(T)v.
\end{equation*}
The convergence \eqref{Eq_dom_ran_density_5} used for the vector $u=f(T)v$, then gives the limit
\begin{equation}\label{Eq_fT_as_closure_5}
\lim\limits_{n\rightarrow\infty}(fe)(T)e(T)^{-1}v_n=\lim\limits_{n\rightarrow\infty}r_n[T]^mf(T)v=f(T)v.
\end{equation}
Finally, combining the limits \eqref{Eq_fT_as_closure_3} and \eqref{Eq_fT_as_closure_5}, shows that $v\in\dom\big(\overline{(fe)(T)e(T)^{-1}}\big)$ and
\begin{equation*}
\overline{(fe)(T)e(T)^{-1}}\,v=f(T)v. \qedhere
\end{equation*}
\end{proof}

\begin{thm}[Reformulation of the right $\omega$-functional calculus]\label{thm_Right_omega_as_closure}
Let $V$ be a reflexive Banach module and $T\in\mathcal{K}(V)$ injective, bisectorial of angle $\omega\in(0,\frac{\pi}{2})$. Then for every $f\in\mathcal{SH}_R^0(D_\theta)$, $\theta\in(\omega,\frac{\pi}{2})$, the operator $(fe)(T)e(T)^{-1}$ is closable, and the right $\omega$-functional calculus \eqref{Eq_Right_omega} can be written as
\begin{equation}\label{Eq_Right_omega_as_closure}
f(T)=\overline{(fe)(T)e(T)^{-1}}.
\end{equation}
Here $e(s)=\frac{s^m}{(1+s^2)^m}$, for some arbitrary $m\in\mathbb{N}$.
\end{thm}

\begin{proof}
For the inclusion $\text{\grqq}\subseteq\text{\grqq}$ we note that due to the product rule $(fe)(T)=f(T)e(T)$ of the $\omega$-functional calculus, there is
\begin{equation}\label{Eq_Right_omega_as_closure_1}
(fe)(T)e(T)^{-1}=f(T)e(T)e(T)^{-1}\subseteq f(T).
\end{equation}
Since $f(T)$ is bounded and hence closed, there is $(fe)(T)e(T)^{-1}$ closable, and
\begin{equation*}
\overline{(fe)(T)e(T)^{-1}}\subseteq f(T).
\end{equation*}
For the inclusion $\text{\grqq}\supseteq\text{\grqq}$ let $v\in V$. With the operators $r_n[T]$ from \eqref{Eq_dom_ran_density_5} we then define the approximating sequence
\begin{equation*}
v_n:=r_n[T]^mv.
\end{equation*}
We then know from \eqref{Eq_fT_as_closure_1} and \eqref{Eq_dom_ran_density_5}, that
\begin{equation}\label{Eq_Right_omega_as_closure_2}
v_n\in\dom((fe)(T)e(T)^{-1})\qquad\text{and}\qquad\lim\limits_{n\rightarrow\infty}v_n=v.
\end{equation}
Using once more the already derived operator inclusion \eqref{Eq_Right_omega_as_closure_1}, as well as the fact that $f(T)$ is a bounded operator, there then also converges
\begin{equation}\label{Eq_Right_omega_as_closure_3}
\lim\limits_{n\rightarrow\infty}(fe)(T)e(T)^{-1}v_n=\lim\limits_{n\rightarrow\infty}f(T)v_n=f(T)v.
\end{equation}
The two limits \eqref{Eq_Right_omega_as_closure_2} and \eqref{Eq_Right_omega_as_closure_3} now show that $v\in\dom\big(\overline{(fe)(T)e(T)^{-1}}\big)$, which proves also the second inclusion.
\end{proof}

The representations \eqref{Eq_fT_as_closure} and \eqref{Eq_Right_omega_as_closure} will now serve as a motivation for the $H^\infty$-functional calculus of right slice hyperholomorphic functions. Note that in \eqref{Eq_fT_as_closure} and \eqref{Eq_Right_omega_as_closure} the respective right hand sides look formally the same as in the upcoming definition \eqref{Eq_Right_Hinfty}. However, for those special functions $f$ it turned out that that the operator $(fe)(T)e(T)^{-1}$ was closable and the closure was understood in the sense of operators. Since for general $f$ this is no longer the case, we have to interpret the closure in \eqref{Eq_Right_Hinfty} in the sense of multivalued operators, which consequently means that the right $H^\infty$-functional calculus $f(T)$ is not an operator but rather a multivalued operator in the sense of Definition \ref{defi_Multivalued_operators}.

\begin{defi}[Right $H^\infty$-functional calculus]\label{defi_Right_Hinfty}
Let $V$ be a $\mathbb{R}$-reflexive Banach module and
$T\in\mathcal{K}(V)$ injective, bisectorial of angle $\omega\in(0,\frac{\pi}{2})$. Then for every $f\in\mathcal{SH}_R^\poly(D_\theta)$, $\theta\in(\omega,\frac{\pi}{2})$ we define the \textit{right $H^\infty$-functional calculus} as
\begin{equation}\label{Eq_Right_Hinfty}
f(T):=\overline{(fe)(T)e(T)^{-1}},
\end{equation}
using the function $e(s)=\frac{s^m}{(1+s^2)^m}$, for some $m\in\mathbb{N}$ with $m>\alpha$ and $\alpha$ from Definition~\ref{defi_SH_poly}. Moreover, the closure is understood in the sense of multivalued operators \eqref{Eq_Closure}.
\end{defi}

\begin{thm}\label{thm_Hinfty_independence}
The right $H^\infty$-functional calculus in Definition~\ref{defi_Right_Hinfty} is well defined. In particular this means that \medskip

\begin{enumerate}
\item[i)] $f(T)$ is independent of the chosen integer $m>\alpha$; \medskip

\item[ii)] $f(T)$ coincides with the left $H^\infty$-functional calculus \eqref{Eq_Left_Hinfty} for $f\in\mathcal{N}^\poly(D_\theta)$; \medskip

\item[iii)] $f(T)$ coincides with the right $\omega$-functional calculus \eqref{Eq_Right_omega} for $f\in\mathcal{SH}_R^0(D_\theta)$.
\end{enumerate}
\end{thm}

\begin{proof}
ii),\,iii)\;\;The fact that $f(T)$ coincides for $f\in\mathcal{N}^\poly(D_\theta)$ with the left $H^\infty$-functional calculus, and for $f\in\mathcal{SH}_R^0(D_\theta)$ with the right $\omega$-functional calculus, was proven in Theorem~\ref{thm_fT_as_closure} and Theorem~\ref{thm_Right_omega_as_closure}. \medskip

i)\;\;In order to prove that \eqref{Eq_Right_Hinfty} is independent of the choice of $m$, let us consider the function $e(s)=\frac{s}{1+s^2}$ and integers $m_1,m_2\in\mathbb{N}$, with $m_1\geq m_2>\alpha$. We will prove that
\begin{equation}\label{Eq_Right_Hinfty_1}
\overline{(fe^{m_1})(T)e(T)^{-m_1}}=\overline{(fe^{m_2})(T)e(T)^{-m_2}}.
\end{equation}
The inclusion $\text{\grqq}\subseteq\text{\grqq}$ follows immediately from the product rule of the $\omega$-functional calculus \cite[Theorem~3.11]{MS24}, namely
\begin{align}
(fe^{m_1})(T)e(T)^{-m_1}&=(fe^{m_2}e^{m_1-m_2})(T)e(T)^{-(m_1-m_2)}e(T)^{-m_2} \notag \\
&=(fe^{m_2})(T)e(T)^{m_1-m_2}e(T)^{-(m_1-m_2)}e(T)^{-m_2} \notag \\
&\subseteq(fe^{m_2})(T)e(T)^{-m_2}. \label{Eq_Right_Hinfty_4}
\end{align}
Taking the closures of both sides gives the first inclusion in \eqref{Eq_Right_Hinfty_1}. \medskip

For the inverse inclusion $\text{\grqq}\supseteq\text{\grqq}$ in \eqref{Eq_Right_Hinfty_1} let $v\in\dom\big((fe^{m_2})(T)e(T)^{-m_2}\big)$, i.e.
\begin{equation*}
v\in\ran(e(T)^{m_2})=\dom(T^{m_2})\cap\ran(T^{m_2}).
\end{equation*}
With the operators $r_n[T]$ in \eqref{Eq_dom_ran_density_2} we now define the vectors $v_n:=r_n[T]^{m_1}v$, $n\in\mathbb{N}$. Then by \eqref{Eq_dom_ran_density_5} and the fact that $\overline{\ran}(T)=V$ by Theorem~\ref{thm_Domain_dense} ii), there converges
\begin{equation}\label{Eq_Right_Hinfty_2}
\lim\limits_{n\rightarrow\infty}v_n=v.
\end{equation}
By \eqref{Eq_fT_as_closure_1} there is $v_n\in\dom\big((fe^{m_1})(T)e(T)^{-m_1}\big)$, which means that with the operator inclusions \eqref{Eq_Right_Hinfty_4} and with $e(T)^{-m_2}r_n[T]^{m_1}\supseteq r_n[T]^{m_1}e(T)^{-m_2}$, we can write
\begin{equation*}
(fe^{m_1})(T)e(T)^{-m_1}v_n=(fe^{m_2})(T)e(T)^{-m_2}v_n=(fe^{m_2})(T)r_n[T]^{m_1}e(T)^{-m_2}v.
\end{equation*}
Using once more the convergence
\begin{equation*}
\lim\limits_{n\rightarrow\infty}r_n[T]^{m_1}e(T)^{-m_2}v=e(T)^{-m_2}v,
\end{equation*}
from \eqref{Eq_dom_ran_density_5} with the value $u=e(T)^{-m_2}v$, we conclude from the boundedness of the operator $(fe^{m_2})(T)$, that there converges
\begin{equation}\label{Eq_Right_Hinfty_3}
\lim\limits_{n\rightarrow\infty}(fe^{m_1})(T)e(T)^{-m_1}v_n=(fe^{m_2})(T)e(T)^{-m_2}v.
\end{equation}
The two limits \eqref{Eq_Right_Hinfty_2} and \eqref{Eq_Right_Hinfty_3} now show that
\begin{equation*}
\big(v\,,\,(fe^{m_2})(T)e(T)^{-m_2}v\big)\in\overline{(fe^{m_1})(T)e(T)^{-m_1}},
\end{equation*}
where we used the notation of multivalued operators, since $\overline{(fe^{m_1})(T)e(T)^{-m_1}}$ on the right hand side is a right linear subspace of $V\times V$. This proves that the graph of the operator $(fe^{m_2})(T)e(T)^{-m_2}$ is contained in
\begin{equation*}
\text{graph}\big((fe^{m_2})(T)e(T)^{-m_2}\big)\subseteq\overline{(fe^{m_1})(T)e(T)^{-m_1}},
\end{equation*}
and hence also the closure (in the sense of multivalued operators) is contained in
\begin{equation*}
\overline{(fe^{m_2})(T)e(T)^{-m_2}}\subseteq\overline{(fe^{m_1})(T)e(T)^{-m_1}}. \qedhere
\end{equation*}
\end{proof}

Let us now collect some basic properties of the right $H^\infty$-functional calculus \eqref{Eq_Right_Hinfty}.

\begin{lem}
Let $V$ be a $\mathbb{R}$-reflexive Banach module and
$T\in\mathcal{K}(V)$ injective, bisectorial of angle $\omega\in(0,\frac{\pi}{2})$. Then for every $f\in\mathcal{SH}_L^\poly(D_\theta)$, $\theta\in(\omega,\frac{\pi}{2})$ the left $H^\infty$-functional calculus \eqref{Eq_Left_Hinfty} has the following properties:

\begin{enumerate}
\item[i)] $f(T)$ is a closed, right linear operator; \medskip

\item[ii)] For every $g\in\mathcal{SH}_L^0(D_\theta)$ and $a\in\mathbb{R}_n$, it satisfies the linearity properties
\begin{equation}\label{Eq_Linearity_Left_Hinfty}
(f+g)(T)=f(T)+g(T)\qquad\text{and}\qquad(fa)(T)=f(T)a.
\end{equation}
\end{enumerate}

Moreover, for every $f\in\mathcal{SH}_R^\poly(D_\theta)$, $\theta\in(\omega,\frac{\pi}{2})$, the right $H^\infty$-functional calculus \eqref{Eq_Right_Hinfty} has the following properties:

\begin{enumerate}
\item[i')] $f(T)$ is a closed, right linear multivalued operator; \medskip

\item[ii')] For every $g\in\mathcal{SH}_R^0(D_\theta)$ and $a\in\mathbb{R}_n$, it satisfies the linearity properties
\begin{equation}\label{Eq_Linearity_Right_Hinfty}
(f+g)(T)=f(T)+g(T)\qquad\text{and}\qquad(af)(T)=af(T).
\end{equation}
\end{enumerate}
\end{lem}

\begin{proof}
Throughout this proof lat $e(s)=\frac{s^m}{(1+s^2)^m}$ be a regularizer of $f$ according to Definition~\ref{defi_Left_Hinfty} or Definition~\ref{defi_Right_Hinfty} respectively. \medskip

i)\;\;Since the left $H^\infty$-functional calculus \eqref{Eq_Left_Hinfty} is the product of a closed and a bounded right linear operator, $f(T)$ is a closed and right linear operator as well. \medskip

ii)\;\;Since $g\in\mathcal{SH}_L^0(D_\theta)$, the function $e$ is a regularizer of the sum $f+g$ as well. It then follows from the definition \eqref{Eq_Left_Hinfty} of the left $H^\infty$-functional calculus, that
\begin{align*}
(f+g)(T)&=e(T)^{-1}(e(f+g))(T)=e(T)^{-1}((ef)(T)+e(T)g(T)) \\
&=e(T)^{-1}(ef)(T)+g(T)=f(T)+g(T).
\end{align*}
The second part of the linearity in \eqref{Eq_Linearity_Left_Hinfty} follows from
\begin{equation*}
(fa)(T)=e(T)^{-1}(fa)(T)=e(T)^{-1}(ef)(T)a=f(T)a.
\end{equation*}
i')\;\;$f(T)$ is defined as a closure and hence a closed, right linear multivalued operator. \medskip

ii')\;\;The second part of the linearity in \eqref{Eq_Linearity_Right_Hinfty} follows directly from the definition \eqref{Eq_Right_Hinfty}, namely
\begin{equation*}
(af)(T)=\overline{(afe)(T)e(T)^{-1}}=a\,\overline{(fe)(T)e(T)^{-1}}=af(T).
\end{equation*}
For the additive linearity in \eqref{Eq_Linearity_Left_Hinfty}, we first note that due to $g\in\mathcal{SH}_R^0(D_\theta)$, the function $e$ is a regularizer of the sum $f+g$ as well. In order to prove $\text{\grqq}\subseteq\text{\grqq}$, we know from the linearity of the $\omega$-functional calculus, that
\begin{align*}
((f+g)e)(T)e(T)^{-1}&=((fe)(T)+g(T)e(T))e(T)^{-1} \\
&=(fe)(T)e(T)^{-1}+g(T)e(T)e(T)^{-1} \\
&\subseteq\overline{(fe)(T)e(T)^{-1}}+g(T)=f(T)+g(T).
\end{align*}
Since the right hand side is the sum of the closed multivalued operator $f(T)$ and the bounded operator $g(T)$, it is a closed multivalued operator as well. Hence there follows
\begin{equation*}
(f+g)(T)=\overline{((f+g)e)(T)e(T)^{-1}}\subseteq f(T)+g(T).
\end{equation*}
For the inverse inclusion $\text{\grqq}\supseteq\text{\grqq}$ let $(v,w)\in f(T)+g(T)$. Since $g(T)$ is a bounded operator, this means by the definition of the sum of multivalued operators in \eqref{Eq_Sum_product}, that $(v,w-g(T)v)\in f(T)$. Moreover, by the definition of $f(T)$ as the closure \eqref{Eq_Right_Hinfty}, there exists a sequence $(v_n)_n\in\dom\big((fe)(T)e(T)^{-1}\big)=\ran(e(T))$, such that
\begin{equation}\label{Eq_Linearity_Hinfty_1}
v=\lim\limits_{n\rightarrow\infty}v_n\qquad\text{and}\qquad w-g(T)v=\lim\limits_{n\rightarrow\infty}(fe)(T)e(T)^{-1}v_n.
\end{equation}
Since $g(T)$ is bounded, there also converges
\begin{equation}\label{Eq_Linearity_Hinfty_2}
g(T)v=\lim\limits_{n\rightarrow\infty}g(T)v_n=\lim\limits_{n\rightarrow\infty}g(T)e(T)e(T)^{-1}v_n=\lim\limits_{n\rightarrow\infty}(ge)(T)e(T)^{-1}v_n.
\end{equation}
Combining now the second limit in \eqref{Eq_Linearity_Hinfty_1} with the limit \eqref{Eq_Linearity_Hinfty_2}, we get
\begin{equation*}
w=\lim\limits_{n\rightarrow\infty}((fe)(T)+(ge)(T))e(T)^{-1}v_n=\lim\limits_{n\rightarrow\infty}((f+g)e)(T)e(T)^{-1}v_n.
\end{equation*}
Together with the first limit in \eqref{Eq_Linearity_Hinfty_1} this shows that
\begin{equation*}
(v,w)\in\overline{((f+g)e)(T)e(T)^{-1}}=(f+g)(T). \qedhere
\end{equation*}
\end{proof}

\begin{prop}
Let $V$ be a $\mathbb{R}$-reflexive Banach module and $T\in\mathcal{K}(V)$ injective, bisectorial of angle $\omega\in(0,\frac{\pi}{2})$. Then for every $f\in\mathcal{SH}_R^\poly(D_\theta)$, there is
\begin{equation*}
\overline{\dom}(f(T))=V.
\end{equation*}
Here, the domain of the multivalued operator $f(T)$ is understood in the sense of \eqref{Eq_Domain}.
\end{prop}

\begin{proof}
If follows directly from the definition of $f(T)$ in \eqref{Eq_Right_Hinfty}, that
\begin{equation*}
\dom(f(T))\supseteq\dom((fe)(T)e(T)^{-1})=\ran(e(T))=\dom(T^m)\cap\ran(T^m).
\end{equation*}
However, we know from Theorem~\ref{thm_dom_ran_density} and Theorem~\ref{thm_Domain_dense} that the right hand side is dense in $V$, and hence also $\overline{\dom}(f(T))=V$ is dense.
\end{proof}

\begin{thm}[Product rule]
Let $V$ be a $\mathbb{R}$-reflexive Banach module, $T\in\mathcal{K}(V)$ injective, bisectorial of angle $\omega\in(0,\frac{\pi}{2})$, and $\theta\in(\omega,\frac{\pi}{2})$. Then the left $H^\infty$-functional calculus \eqref{Eq_Left_Hinfty} satisfies the product rule

\begin{enumerate}
\item[i)] $(fg)(T)\supseteq f(T)g(T)$,\qquad $f\in\mathcal{N}^0(D_\theta)$, $g\in\mathcal{SH}_L^\poly(D_\theta)$; \medskip

\item[ii)] $(fg)(T)=f(T)g(T)$,\qquad $f\in\mathcal{N}^\poly(D_\theta)$, $g\in\mathcal{SH}_L^0(D_\theta)$.
\end{enumerate}

Moreover, the right $H^\infty$-functional calculus \eqref{Eq_Right_Hinfty} satisfies the product rule

\begin{enumerate}
\item[i')] $(fg)(T)=\overline{f(T)g(T)}$,\qquad $f\in\mathcal{SH}_R^0(D_\theta)$, $g\in\mathcal{N}^\poly(D_\theta)$; \medskip

\item[ii')] $(fg)(T)\subseteq f(T)g(T)$,\qquad $f\in\mathcal{SH}_R^\poly(D_\theta)$, $g\in\mathcal{N}^0(D_\theta)$.
\end{enumerate}
\end{thm}

\begin{proof}
For the whole proof let $e(s)=\frac{s^m}{(1+s^2)^m}$ be a regularizer of $g$ in the cases i), i') and a regularizer of $f$ in the cases ii), ii'). Since the respective other function is bounded in all the cases it turns out that $e$ is always a regularizer of the produce $fg$ as well. \medskip

i)\;\;The commutation property $e(T)f(T)=f(T)e(T)$ from \cite[Corollary~3.18~iii)]{MS24}, leads us to
\begin{equation*}
f(T)e(T)^{-1}=e(T)^{-1}e(T)f(T)e(T)^{-1}=e(T)^{-1}f(T)e(T)e(T)^{-1}\supseteq e(T)^{-1}f(T),
\end{equation*}
and consequently there follows the product rule
\begin{equation*}
(fg)(T)=e(T)^{-1}(efg)(T)=e(T)^{-1}f(T)(eg)(T)\supseteq f(T)e(T)^{-1}(eg)(T)=f(T)g(T).
\end{equation*}
ii)\;\;Using that $g(T)$ is a bounded operator, there follows from the product rule of the $\omega$-functional calculus that
\begin{equation*}
(fg)(T)=e(T)^{-1}(efg)(T)=e(T)^{-1}(ef)(T)g(T)=f(T)g(T).
\end{equation*}
i')\;\;For the inclusion $\text{\grqq}\supseteq\text{\grqq}$, we note that $g$ is intrinsic, which means that by Theorem~\ref{thm_fT_as_closure}, $g(T)$ can also be written by the right $H^\infty$-functional calculus \eqref{Eq_Left_Hinfty}, and we get
\begin{equation*}
f(T)g(T)=f(T)\overline{(ge)(T)e(T)^{-1}}\subseteq\overline{f(T)(ge)(T)e(T)^{-1}}=\overline{(fge)(T)e(T)^{-1}}=(fg)(T),
\end{equation*}
where we used that $B\overline{A}\subseteq\overline{BA}$ if $B$ is a bounded operator and $A$ a multivalued operator. Since the right hand side is closed, we can put the closure on the left hand side also and get
\begin{equation*}
\overline{f(T)g(T)}\subseteq(fg)(T).
\end{equation*}
For the inverse inclusion $\text{\grqq}\subseteq\text{\grqq}$, we have
\begin{equation*}
(fge)(T)e(T)^{-1}=f(T)(ge)(T)e(T)^{-1}\subseteq f(T)\overline{(ge)(T)e(T)^{-1}}=f(T)g(T).
\end{equation*}
Performing the closure on both sides of this inclusion, we get
\begin{equation*}
(fg)(T)=\overline{(fge)(T)e(T)^{-1}}\subseteq\overline{f(T)g(T)}.
\end{equation*}
ii')\;\;For the last product rule we consider $g(T)$ is a bounded operator. Then there is
\begin{equation*}
(fge)(T)e(T)^{-1}=(fe)(T)g(T)e(T)^{-1}\subseteq(fe)(T)e(T)^{-1}g(T)\subseteq f(T)g(T).
\end{equation*}
Since the right hand side is closed, as the product of a closed multivalued operator and a bounded operator, we can also perform the closure of the left hand side and conclude $(fg)(T)\subseteq f(T)g(T)$.
\end{proof}

\begin{defi}
Let $T\in\mathcal{K}(V)$. Then for Clifford numbers $p_0,p_1,\dots,p_N\in\mathbb{R}_n$ and the corresponding right polynomial $p(s)=p_0+p_1s+\dots p_Ns^N$, we define the \textit{right polynomial functional calculus}
\begin{equation*}
p[T]:=p_0+p_1T+\dots+p_NT^N,\qquad\text{with }\dom(p[T]):=\dom(T^N).
\end{equation*}
\end{defi}

\begin{rem}
The natural definition
\begin{equation}\label{Eq_Left_polynomial_functional_calculus}
p[T]=p_0+Tp_1+\dots+T^Np_N,
\end{equation}
of the polynomial functional calculus for left polynomials $p(s)=p_0+sp_1+\dots+s^Np_N$ does not work, since it is not defined on $\dom(T^N)$ and also the expression \eqref{Eq_Left_polynomial_functional_calculus} does not lead a right linear operator.
\end{rem}

\begin{thm}
Let $V$ be a $\mathbb{R}$-reflexive Banach module and
$T\in\mathcal{K}(V)$ injective, bisectorial of angle $\omega\in(0,\frac{\pi}{2})$. Furthermore let $p$ be a right polynomial, and $q$ an intrinsic polynomial which does not admit any zeros in $\overline{D_\omega}$. Then $q[T]$ is bijective, there exists some $\theta\in(\omega,\frac{\pi}{2})$ such that $\frac{p}{q}\in\mathcal{SH}_R^\poly(D_\theta)$, and the right $H^\infty$-functional calculus of $pq^{-1}$ is given by
\begin{equation}\label{Eq_Rational_functional_calculus}
(pq^{-1})(T)=\overline{p[T]q[T]^{-1}}.
\end{equation}
\end{thm}

\begin{proof}
First, it is already shown in \cite[Theorem 5.9]{MS24} that if the intrinsic polynomial $q$ does not admit any zeros in the closed double sector $\overline{D_\omega}$, the polynomial functional calculus $q[T]$ turns out to be bijective. Moreover, since $q$ does not admit any zeros in $\overline{D_\omega}$ there also exists some larger angle $\theta\in(\omega,\frac{\pi}{2})$ such that $q$ does not admit any zeros in $D_\theta$. Since $pq^{-1}$ is clearly polynomially bounded at zero and at infinity, we then clearly have $pq^{-1}\in\mathcal{SH}_R^\poly(D_\theta)$. Let now $e(s)=\frac{s^m}{(1+s^2)^m}$ be a regularizer of $pq^{-1}$ according to Definition~\ref{defi_Right_Hinfty} and we will prove the identity \eqref{Eq_Rational_functional_calculus}. \medskip

For the inclusion $\text{\grqq}\subseteq\text{\grqq}$, we consider $p(s)=\sum_{k=0}^Np_Ns^N$ and use the left-linearity of the right $\omega$-functional, as well as the product rule $(s^kq^{-1}e)(T)=T^k(q^{-1}e)(T)$ of the intrinsic polynomial functional calculus \cite[Proposition 3.15]{MS24}, to write
\begin{equation*}
(pq^{-1}e)(T)=\sum\limits_{k=0}^Np_k(s^kq^{-1}e)(T)=\sum\limits_{k=0}^Np_kT^k(q^{-1}e)(T)=p[T](q^{-1}e)(T).
\end{equation*}
Multiplying this equation from the right with $e(T)^{-1}$, and use the intrinsic rational functional calculus $(eq^{-1})(T)=e[T]q[T]^{-1}$ from \cite[Theorem 3.16]{MS24}, we furthermore get
\begin{equation}\label{Eq_Rational_functional_calculus_1}
(pq^{-1}e)(T)e(T)^{-1}=p[T]q[T]^{-1}e[T]e(T)^{-1}\subseteq p[T]q[T]^{-1}.
\end{equation}
Performing the closure of both sides of this equation then gives the first inclusion
\begin{equation*}
(pq^{-1})(T)=\overline{(pq^{-1}e)(T)e(T)^{-1}}\subseteq\overline{p[T]q[T]^{-1}}.
\end{equation*}
For the inverse inclusion $\text{\grqq}\supseteq\text{\grqq}$ let $v\in\dom(p[T]q[T]^{-1})$, and consider $v_n:=r_n[T]^mv$, using the operators $r_n[T]$ from \eqref{Eq_dom_ran_density_2}. Using the already proven inclusion \eqref{Eq_Rational_functional_calculus_1}, we get
\begin{equation*}
(pq^{-1}e)(T)e(T)^{-1}v_n=p[T]q[T]^{-1}v_n=\sum_{k=0}^Np_kT^kq[T]^{-1}r_n[T]^mv_n=\sum_{k=0}^Np_kr_n[T]^mT^kq[T]^{-1}v_n.
\end{equation*}
We know from \eqref{Eq_dom_ran_density_5} and Theorem~\ref{thm_Domain_dense}~ii), with the vector $u=T^kq[T]^{-1}v$, that there converges
\begin{equation*}
\lim_{n\rightarrow\infty}r_n[T]^mT^kq[T]^{-1}v=T^kq[T]^{-1}v,\qquad k\in\{0,\dots,N\}.
\end{equation*}
Hence there also converges the whole sum
\begin{equation*}
\lim_{n\rightarrow\infty}(pq^{-1}e)(T)e(T)^{-1}v_n=\sum_{k=0}^Np_kT^kq[T]^{-1}v=p[T]q[T]^{-1}v.
\end{equation*}
There clearly also converges $\lim_{n\to\infty}v_n=v$ by \eqref{Eq_dom_ran_density_5}, which then implies that
\begin{equation*}
(v,p[T]q[T]^{-1}v)\in\overline{(pq^{-1}e)(T)e(T)^{-1}}=(pq^{-1})(T).
\end{equation*}
Since this is true for every $v\in\dom(p[T]q[T]^{-1})$, we have proven the inclusion
\begin{equation*}
p[T]q[T]^{-1}\subseteq(pq^{-1})(T),
\end{equation*}
in the sense of multivalued operators. Finally, since $(pq^{-1})(T)$ is closed, we can also close the left hand side and end up with the inclusion $\overline{p[T]q[T]^{-1}}\subseteq(pq^{-1})(T)$.
\end{proof}

\section{Concluding remarks}

The $H^\infty$-functional calculi also have a broader context  in the recently introduced fine structures on the $S$-spectrum. These are function spaces of nonholomorphic functions derived from the Fueter-Sce extension theorem, which connects slice hyperholomorphic and axially monogenic functions, the two main notions of holomorphicity in the Clifford setting. The connection is established through powers of the Laplace operator in a higher dimension, see \cite{Fueter,TaoQian1,Sce} and also the translation \cite{ColSabStrupSce}. \medskip

The fine structures on the $S$-spectrum constitute a set of function spaces and their related functional calculi. These structures have been introduced and studied in recent works \cite{CDPS1,Fivedim,Polyf1,Polyf2}, while their respective $H^\infty$-versions are investigated in \cite{MILANJPETER,MPS23}. \medskip

Unlike complex holomorphic function theory, the noncommutative setting of the Clifford algebra allows multiple notions of hyperholomorphicity for vector fields. Consequently, different spectral theories emerge, each based on the concept of spectrum that relies on distinct Cauchy kernels. The spectral theory based on the $S$-spectrum was inspired by quaternionic quantum mechanics (see \cite{BF}), is associated with slice hyperholomorphicity, and began its development in 2006. A comprehensive introduction can be found in \cite{CGK}, with further explorations done in \cite{ACS2016,AlpayColSab2020,ColomboSabadiniStruppa2011}. Applications on fractional powers of operators are investigated in \cite{CGdiffusion2018,CG18,FJBOOK} and some results from classical interpolation theory, see \cite{BERG_INTER,BRUDNYI_INTER,Ale_INTER,TRIEBEL}, have been recently extended into this setting \cite{COLSCH}. In order to fully appreciate the spectral theory on the $S$-spectrum, we recall that it applies to sets of noncommuting operators, for example with the identification $(T_1,...,T_{2^n})\leftrightarrow T=\sum_AT_Ae_A$ it can be seen as a theory for several operators. \medskip

It is important to highlight that in the hypercomplex setting there exists another spectral theory based on the monogenic spectrum, which was initiated by Jefferies and McIntosh in \cite{JM}. This theory is founded on monogenic functions, which are functions in the kernel of the Dirac operator, and their associated Cauchy formula (see \cite{DSS} for a comprehensive treatment). \medskip

The monogenic spectral theory has been extensively developed and is well described in the seminal works of Jefferies \cite{JBOOK} and Tao \cite{TAOBOOK}. These texts provide an in-depth exploration of the theory, with particular emphasis on the $H^\infty$-functional calculus in the monogenic setting. \medskip

The idea of the corresponding regularization procedure has been introduced by A. McIntosh in \cite{McI1} and further explored in various studies as \cite{MC10,MC97,MC06,MC98}. It is extensively used in the investigation of evolution equations, particularly in maximal regularity and fractional powers of differential operators. Additionally, there exists a close relation to Kato's square root problem and pseudo-differential operators.

\section*{Declarations and statements}

\textbf{Data availability}. The research in this paper does not imply use of data. \medskip

\textbf{Conflict of interest}. The authors declare that there is no conflict of interest.

\end{document}